\documentclass[a4paper, 11pt]{article}

\usepackage{packageList}
\usepackage{macros}

\IfFileExists{bbm.sty}{\usepackage{bbm}%
\newcommand{\deltak}[2]{\mathbbm{1}_{\{##2\}}(##1)}%
}{%
\newcommand{\deltak}[2]{\delta_{##2}(##1)}
}

\title{Greedy regularized kernel interpolation}
\author[1]{G. Santin \thanks{gabriele.santin@mathematik.uni-stuttgart.de, 
orcid.org/0000-0001-6959-1070}}
\author[1]{D. Wittwar \thanks{dominik.wittwar@mathematik.uni-stuttgart.de}}
\author[1]{B. Haasdonk \thanks{bernard.haasdonk@mathematik.uni-stuttgart.de}}
\affil[1]{Institute for Applied Analysis and Numerical Simulation, University of Stuttgart, Germany}

\begin{document}

\maketitle

\begin{abstract}
Kernel based regularized interpolation is a well known technique to approximate a continuous multivariate function using a set of scattered data points 
and the corresponding function evaluations, or data values. This method has some advantage over exact interpolation: one can obtain the same approximation order while 
solving a better conditioned linear system. This method is well suited also for noisy data values, where exact interpolation is not meaningful. Moreover, it allows 
more flexibility in the kernel choice, since approximation problems can be solved also for non strictly positive definite kernels.
We discuss in this paper a greedy algorithm to compute a sparse approximation of the kernel regularized interpolant. This sparsity is a desirable property when the 
approximant is used as a surrogate of an expensive function, since the resulting model is fast to evaluate.
Moreover, we derive convergence results for the approximation scheme, and we prove that a certain greedy selection rule produces asymptotically quasi-optimal 
error rates.
\end{abstract}

\section{Kernels and regularized interpolation}\label{sec:intro}
Our goal is to construct an approximant on an input space $\Omega\subset {\mathbb R}^d$, $d\geq 1$, of an unknown continuous function $f:\Omega\to{\mathbb R}$ provided 
the 
knowledge 
of arbitrary pairwise distinct data points $X_n:=\{x_i\}_{i=1}^n\subset\Omega$, $n\in\N$, and data values $\{f(x_i)\}_{i=1}^n\subset {\mathbb R}$.

The approximant is constructed via kernel interpolation. We recall here the basic facts required for our analysis, while we refer to \cite{Wendland2005} for further 
details. 

On $\Omega$ we consider a positive definite kernel 
$K:\Omega\times\Omega\to 
{\mathbb R}$, i.e., a symmetric function such that for any 
$n\in\N$ and 
any set
$X_n:=\{x_i\}_{i=1}^n\subset \Omega$ of pairwise distinct points the {{kernel matrix}} $A\in{\mathbb R}^{n\times n}$, $A_{ij} := K(x_i, x_j)$, is 
positive 
semidefinite. For a {strictly} positive definite $K$,  $A$ is required to be positive definite.

Associated with the kernel $K$ there is a uniquely defined {{native}} Hilbert space $\calh:=\ns$ of functions $\Omega\to{\mathbb R}$. The space is the unique 
Hilbert space of functions from $\Omega$ to ${\mathbb R}$ where $K$ acts as a reproducing kernel, i.e., $K(\cdot, x)\in\calh$ 
and $(f, K(\cdot, 
x))_{\calh} = f(x)$ for all $x\in\Omega$, $f\in\calh$. 
The elements of this space  are of the form $f:=\sum_{i\in I} \alpha_i K(\cdot, x_i)$ for $I$ a countable set, $\{\alpha_i\}_{i\in I} \subset {\mathbb R}$, and 
$\{x_i\}_{i\in 
I}\subset\Omega$, and for $g:=\sum_{j\in J} \beta_j K(\cdot, y_j)$ it holds 
$
(f, g)_{\calh} = \sum_{i\in I, j \in J}\alpha_i \beta_j K(x_i, y_j).
$
Moreover, if the kernel has smoothness $K\in C^{2\tau}(\Omega\times\Omega)$ with $\tau\geq0$ and $\Omega$ an open set, then it holds that $\ns\subset C^{\tau}(\Omega)$.

One of the main reasons of interest for positive definite kernels is that various approximation problems can be solved in $\calh$ for arbitrary pairwise distinct 
data points $X_n\subset\Omega$ and data values $\{f(x_i)\}_{i=1}^n\subset {\mathbb R}$, $f\in\calh$.
Indeed, one can consider a {loss} and a {regularization} functional $\mathcal L, \mathcal R:\calh\to {\mathbb R}$, and a regularization parameter $\lambda\geq 0$ 
to define an approximant of $f\in\calh$ as  
\begin{equation}\label{eq:representer}
s_n^{\lambda}(f):=s^{\lambda}(f, X_n):=\argmin_{s\in\calh} \mathcal L(s)  + \lambda \mathcal R(s), 
\end{equation}
and to obtain a pointwise reconstruction of $f$ it is common to consider functionals defined as
\begin{equation}\label{eq:functionals}
\mathcal L(s):= \mathcal L(s, f, X_n):=\sum_{i=1}^n \left(f(x_i) - s(x_i)\right)^2,\;\; \mathcal R(s):= \|s\|_{\calh}^2, 
\end{equation}
in which case a solution $s_n^{\lambda}(f)$ of \cref{eq:representer} is called a {{regularized interpolant}} of $f$. Other choices of the functionals lead, for example, 
 to Support Vector Machines and Support Vector Regression (see e.g. \cite{Steinwart2008a}).

This approximation process is well known and characterized for a wide class of functionals by the Representer Theorem (see \cite{Wahba1999}, and \cite{Schoelkopf2001v} 
for a general statement), and for the special case considered here the following holds. 
\begin{theorem}[Representer Theorem for regularized interpolation]\label{th:representer}
If $K$ is positive definite, the problem \cref{eq:representer} with functionals \cref{eq:functionals} admits a solution of the form
\begin{equation}\label{eq:ansatz}
s_n^{\lambda}(f) = \sum_{j=1}^n \alpha_j K(\cdot, x_j), 
\end{equation}
where the vector of coefficients $\alpha\in{\mathbb R}^n$ is the solution of the linear system
\begin{equation}\label{eq:reg_A}
\left(A + \lambda I\right) \alpha = b,\;\; b_i :=f(x_i).
\end{equation}
If $K$ is strictly positive definite, this is the unique solution for all $\lambda\geq 0$.
\end{theorem}

This approximant is well defined also for positive definite kernels, since the linear system \cref{eq:reg_A} has a unique solution as long as $\lambda>0$. 
This is not the case for pure interpolation, i.e., $\lambda=0$, since the matrix $A$ can be singular in this case.
On the other hand, in the case of strictly positive definite kernels the interpolant $s^0(f)$ always exists and is unique, but it is still  
in general useful to consider a regularized 
interpolant. Indeed, $\lambda$ is a tunable parameter which provides a trade-off between pointwise accuracy, since $s^0(f)$ exactly 
interpolates $f$ on $X_n$, and stability, since the condition number of $A+\lambda I$ is a strictly decreasing function of $\lambda$. Moreover, in several applications 
the data values $\{f(x_i)\}_{i=1}^n$ may be affected by noise, thus it makes no sense to require exact interpolation.

We remark that this kind of approximation can be extended to deal with vector-valued functions $f:\Omega\to{\mathbb R}^q$, $q>1$. In this case an approximant can be 
easily 
obtained by applying the same scheme to each of the $q$ components of $f$, while 
keeping the set $X_n$ fixed across them. The only required modification is to change the right hand side in \cref{eq:reg_A}, which becomes a $n\times q$ matrix with 
$b_i := f(x_i)^T$. The resulting solution $\alpha$ is now also a $n\times q$ matrix, and each of its rows can be used as a coefficient vector in 
\cref{eq:representer} to obtain the desired vector-valued prediction. This construction corresponds to the use of a trivial {matrix-valued kernel}, but more 
sophisticated options are possible (see e.g. \cite{Micchelli2005,WSH17a}). Nevertheless, we consider here only the case $q=1$, while we will analyze the 
general matrix-valued case in full generality in a forthcoming work.

The goal of this paper is to describe an efficient way to compute $s_n^{\lambda}(f)$ for an iteratively increasing set of points $X_n$, which is
adaptively enlarged at each iteration by selecting a new point from a set $\Omega_h\subset\Omega$ in a greedy way. 
This method is a direct extension of the (Vectorial) Kernel Orthogonal Greedy Algorithm ((V)KOGA) \cite{Wirtz2013}, which applies to the case of exact 
interpolation (i.e., 
$\lambda=0$) with strictly positive definite kernels. We will describe this extension and the resulting algorithm in \cref{sec:iterative}, and we will 
consider greedy selection rules of $X_n$ which generalize the $f$-  and $P$-greedy rules for interpolation (\cite{DeMarchi2005, SchWen2000}).

When the points are selected freely inside $\Omega$, i.e., $\Omega_h:=\Omega$, the process is a way to place suitable sampling points $X_n$. If instead 
the selection is made from a large but finite set $\Omega_h:=X_N\subset \Omega$ of given data points or measure locations with $N\gg n$, then $s_n^{\lambda}(f)$ can be 
understood as  a sparse 
approximation of $s_N^{\lambda}(f)$, in the sense that in the sum \cref{eq:ansatz} only the terms corresponding to points in $X_n$ are nonzero (although the coefficients 
are in general not the same). In both cases, a good selection of $X_n$ guarantees that only a small number $n$ of points is sufficient to obtain a good 
accuracy.

The reduction of the number of non-zero terms in the expansion \cref{eq:ansatz} has different computational advantages, and it is mainly interesting in case 
$s_n^{\lambda}(f)$ is used as a surrogate model of an expensive function $f$ in a multi-query scenario (see e.g. \cite{Koeppel2017,KSHH2017,Wirtz2015a}). In this case, 
the time required to obtain the evaluation $s_n^{\lambda}(f)(x)$  for a new input $x\in{\mathbb R}^d$ is a crucial measure of the usability of the surrogate, and it 
clearly 
depends 
on the size of the expansion \cref{eq:ansatz}.

In some notable cases, also convergence rates can be derived for the regularized interpolation process using \textit{sampling inequalities} 
\cite{WendlandRieger2005,Rieger2008b}. They apply to translational invariant kernels such as the Gaussian or the Wendland kernels \cite{Wendland1995a}, and they prove 
that regularized interpolation has the same error rate of interpolation, provided $\lambda$ is chosen small enough, depending on the distribution of the interpolation 
points. We will adapt them to our generalized setting and, after proving some general error bounds in \cref{sec:error}, in \cref{sec:convergence} we will 
show that, in the case of the generalization of $P$-greedy, the results of \cite{SH16b} can be extended to conclude that the greedy selected points provide the same 
convergence rate given by these sampling inequalities for optimally placed points.

\section{Iterative computation and greedy algorithms}\label{sec:iterative}
The general structure of the greedy algorithm is the following. 
We will come back in \cref{sec:greedy_selections} to good criteria to select the next point $x_n$, and for now we 
concentrate on the computation of $s^{\lambda}_n(f)$, $f\in\calh$. We start from the empty set $X_0:=\emptyset$, the zero subspace $V(X_0):=\{0\}$, and the zero 
interpolant $s^{\lambda}_0(f) := 0 \in V(X_0)$. At every iteration $n>0$ we select a new point $x_n\in \Omega_h \setminus X_{n-1}$ and define $X_n:=X_{n-1}\cup\{x_n\}$ 
and $V(X_n):=\Sp{K(\cdot, x_i), x_i\in X_n}$, and then compute $s^{\lambda}_n(f) \in V(X_n)$ by \cref{th:representer} as the regularized interpolant 
with 
data 
points $X_n$ and values $\{f(x_i), x_i\in X_n\}$.

Since $s^{\lambda}_n(f)\in V(X_n)$, for any basis $\{v_k\}_{k=1}^n$ of $V(X_n)$ we can write $s^{\lambda}_n(f)
= \sum_{k=1}^n c_k v_k$ for suitable coefficients $\{c_k\}_{k=1}^n$. To have an efficient computation of $s_n^{\lambda}(f)$ and to avoid recomputing already computed 
quantities, we should employ a {nested} basis, i.e., $\Sp{v_k}_{k=1}^n = V(X_n)$ for all $n$, and have that the coefficients $\{c_k\}_{k=1}^{n-1}$ do 
not change at step $n$.

In the case of non regularized interpolation with a strictly positive definite kernel $K$, the basis satisfying these properties is the Newton basis of \cite{Muller2009, 
Pazouki2011}, which can be obtained by a Gram-Schmidt orthonormalization of $\{K(\cdot, x_i), x_i\in X_n\}$ in $\calh$, and for which it holds 
\begin{equation}\label{eq:basis_definition}
v_k:=\sum_{j=1}^n \beta_{jk}K(\cdot, x_j), \; 1\leq k\leq n,
\end{equation}
with a matrix of coefficients $C_v:=[\beta_{jk}]_{j,k=1}^n = L^{-T}$, where $A = LL^T$ is the Cholesky factorization of the kernel matrix $A$. This basis can be easily 
updated when adding a new point, since the leading principal submatrix of $L$ is the Cholesky factor of the corresponding leading principal submatrix of $A$. The 
resulting VKOGA algorithm uses this basis, suitable selection rules for the new point, and the extension to vector-valued functions outlined in Section 
\cref{sec:intro}

To extend this construction to the case of regularized interpolation for possibly non strictly positive definite kernels, from \cref{th:representer} we 
see that the regularized interpolant is defined by coefficients which solve a linear system with matrix $A + \lambda I$. This matrix is in fact the kernel 
matrix of the kernel $K_{\lambda}(x, y):=K(x,y) + \lambda \deltak{x}{y}$ on the points $X_n$, where $\deltak{x}{y}$ is the indicator function of the set 
$\{x\}$, which is clearly symmetric and it is indeed strictly positive 
definite for $\lambda>0$. In the following proposition we 
prove this fact and some related properties of the corresponding native spaces.

\begin{proposition}\label{prop:hl}
Let $\Omega\subset{\mathbb R}^d$ have non empty interior, $K\in C(\Omega\times\Omega)$ be a positive definite kernel on $\Omega$, and $\lambda> 0$.
Then 
\begin{enumerate}[i)]
\item\label{item1} $K'(x, y):= \lambda \deltak{x}{y}$ and $K_{\lambda}(x, y):= K(x, y) + K'(x, y)$ are strictly positive definite kernels on $\Omega$.
\item\label{item2} The native spaces are related by $\nsl = \calh_{K}(\Omega) \oplus \calh_{K'}(\Omega)$.
\item\label{item3} For all $f\in\nsl$ there exist unique $g\in\ns$, $h\in \calh_{K'}(\Omega)$ such that $f=g+h$, and it holds $\|f\|_{\nsl}^2 = 
\|g\|_{\ns}^2+\|h\|_{\calh_{K'}(\Omega)}^2$.
\end{enumerate}
\end{proposition}
\begin{proof}
For any set $X_n\subset\Omega$ the kernel matrix of $K'$ is the scaled identity matrix $\lambda\cdot I$, so $K'$ is clearly positive definite and strictly 
positive definite if $\lambda>0$. Its native space consists of functions $f(x):=\sum_{i\in I} \alpha_i K'(x, x_i) = \lambda \sum_{i\in I} \alpha_i \delta_{x_i}(x)$ for a 
countable set $I$, 
$\{x_i\}_{i\in I}\subset\Omega$, and $\{\alpha_i\}_{i\in I}\subset {\mathbb R}$. In particular, since $\{x_i\}_{i\in I}$ is countable and $\Omega$ is more than 
countable, 
$\calh_{K'}(\Omega)$ contains no continuous functions except for $f:=0$.

Since $K$, $K'$ are positive definite, according to \cite[Section 6]{Aronszajn1950} also $K_{\lambda}$ is positive definite, and it is strictly positive definite if at 
least one between $K$ and $K'$ is strictly positive definite, so in particular for $\lambda>0$. 

Since $K\in C(\Omega\times\Omega)$ it follows that $\ns\subset C(\Omega)$ and thus $\ns \cap \calh_{K'}(\Omega) = \{0\}$. Then again \cite[Section 
6]{Aronszajn1950} guarantees that \cref{item2} and \cref{item3} hold.

\end{proof}
For simplicity we use from now on the notation $\calh:=\ns$ and $\calhl:=\nsl$.
From the last proposition, \cref{item1}, it follows that for $\lambda>0$, $X_n\subset\Omega$, and $f\in\calhl$,  the 
$\calhl$-interpolant of $f$ on $X_n$ is well defined. 
We denote it as 
$$
I_n^{\lambda}(f) = \sum_{j=1}^n \alpha_j K_{\lambda}(\cdot, x_j),
$$
where $(A + \lambda I)\alpha = b$. Using \cref{item3} of \cref{prop:hl} and the definition of $s_n^{\lambda}(f)$, we have a unique decomposition of 
$I_n^{\lambda}(f)$, which needs to satisfy
\begin{align*}
I_n^{\lambda}(f)(x) &= \sum_{j=1}^n \alpha_j K_{\lambda}(x, x_j)  = \sum_{j=1}^n \alpha_j K(x, x_j) + \sum_{j=1}^n \alpha_j \lambda \deltak{x}{x_j}\\
&= s_n^{\lambda}(f)(x) + \sum_{j=1}^n \alpha_j \lambda \deltak{x}{x_j},\;\; x\in\Omega.
\end{align*}
Observe that this construction implies that the regularized interpolant is well defined also for $f\in\calhl$. Moreover, for all $f\in\calhl$ we get $I_n^{\lambda}(f)(x) 
= s_n^{\lambda}(f)(x)$ 
if $x\notin X_n$.

The same decomposition remains valid if $I_n^{\lambda}(f)$ is expressed in terms of the Newton basis of $X_n$ in $\calhl$, which we denote as 
$\{v_k^{\lambda}\}_{k=1}^n$, 
and which is defined, analogous to \cref{eq:basis_definition}, by coefficients 
\begin{equation}\label{eq:change_of_basis}
C_v := L^{-T},\;\;A + \lambda I = L\ L^{T}. 
\end{equation}
Once again, we recall that 
this basis exists since $K^{\lambda}$ is strictly 
positive definite by \cref{item1} of \cref{prop:hl}. We recall that the interpolant $I_n^{\lambda}(f)$ is the orthogonal projection of 
$f\in\calhl$ into $\Sp{K_{\lambda}(\cdot, x_i), x_i\in X_n}$, and since the basis is orthonormal it holds
\begin{equation}\label{eq:interp_in_calhl}
 I_n^{\lambda}(f) = \sum_{k=1}^n (f, v_k^{\lamdba})_{\calhl} v_k^{\lamdba}. 
\end{equation}
Morever, as elements of $\calhl$,  also the functions $v_k^{\lamdba}$ have a unique decomposition, 
which is 
\begin{align}\label{eq:v_k_l}
v_k^{\lambda}(x):= \sum_{j=1}^n \beta_{jk} K^{\lambda}(x, x_j) = \sum_{j=1}^n \beta_{jk} K(x, x_j) + \lambda \sum_{j=1}^n \beta_{jk} \deltak{x}{x_j},\;\; x\in\Omega..
\end{align}
For every $1\leq k\leq n$, we denote as $v_k$ the first term in the right hand side, and as in the case of the interpolant above we have $v_k^{\lambda}(x) = 
v_k(x)$ if $x\notin X_n$. The elements $\{v_k\}_{k=1}^n$ are clearly in $V(X_n)$. If $K$ is strictly positive definite they are linearly independent since the 
matrix $C_v$ from \cref{eq:change_of_basis} is invertible, so they are a basis, while they are at least a generating set for 
$V(X_n)$ if $K$ is only positive definite. This set of functions is what we need to have the efficient update of the regularized interpolant.

\begin{proposition}\label{prop:newton}
Let $X_{n}:=X_{n-1}\cup\{x_n\}\subset \Omega$ and $f\in \calhl$. Then it holds
\begin{equation}\label{eq:reg_int1}
s_n^{\lambda}(f) = \sum_{k=1}^n (f, v_k^{\lambda})_{\calhl} v_k = s_{n-1}^{\lambda}(f) + (f, v_n^{\lambda})_{\calhl} v_n.
\end{equation}
Moreover, if $f\in\calh$ it holds $(f, v_k)_{\calh} = (f, v_k^{\lambda})_{\calhl}$ so
\begin{equation}\label{eq:reg_int2}
s_n^{\lambda}(f) = \sum_{k=1}^n (f, v_k)_{\calh} v_k = s_{n-1}^{\lambda}(f) + (f, v_n)_{\calh} v_n.
\end{equation}
\end{proposition}
\begin{proof}
For the first equality in \cref{eq:reg_int1} we just need to prove that $\sum_{k=1}^n (f, v_k^{\lambda})_{\calhl} v_k$ equals \cref{eq:ansatz} with $\alpha$ that 
satisfies \cref{eq:reg_A}. This holds since
\begin{align*}
\sum_{k=1}^n (f, v_k^{\lambda})_{\calhl} v_k &= \sum_{k=1}^n \sum_{i=1}^n \beta_{ik} \left(f, K_{\lambda}(\cdot, x_i) \right)_{\calhl} \sum_{j=1}^n \beta_{jk} K(\cdot, 
x_j)\\
&=\sum_{j=1}^n  \left(\sum_{i=1}^n \sum_{k=1}^n  \beta_{ik}
\beta_{jk} f(x_i)\right)  K(\cdot, x_j),
\end{align*}
and $\beta_{ik} = (C_v)_{ik} = (L^{-T})_{ik}$, thus
\begin{align*}
\sum_{i=1}^n \sum_{k=1}^n  \beta_{ik} \beta_{jk}f(x_i)  &=  \sum_{i=1}^n \left(\sum_{k=1}^n  (L^{-T})_{ik}
(L^{-T})_{jk}\right) f(x_i)  \\
& = \sum_{i=1}^n \left((A+\lambda I)^{-1}\right)_{ji} f(x_i)  = \alpha_j.
\end{align*}
The second equality holds because the basis is nested since $v_n^{\lambda}$ depends only on the 
points $X_n$, being $C_v$ upper triangular.
Moreover, if $f\in\calh\subset\calhl$ it holds $(f, K(\cdot, x))_{\calh} = (f, K_{\lambda}(\cdot, x))_{\calhl}$ since both are the reproducing kernel of the 
corresponding space, so in particular $(f, v_k)_{\calh} = (f, 
v_k^{\lambda})_{\calhl}$ holds by linearity and \cref{eq:reg_int2} follows.
\end{proof}

In the case of non regularized interpolation it is also possible to define a residual as the difference between the target function $f$ and the current interpolant. It 
can be used to decide what point to select in the $f$-greedy variant of VKOGA, and it also has an efficient update formula. 

The same can be obtained for regularized interpolation as follows. Observe that both the residual and the update rule are nothing but the ones of the interpolant 
$I_n^{\lambda}(f)$, which are already known to satisfy the desired properties. We prove the statement only for completeness.
\begin{proposition}\label{prop:residual}
Let $f\in\calhl$ and define the residual $r_n\in\calhl$ as 
\begin{equation}\label{eq:update}
r_0:=f,\;\; r_n:= f - \sum_{k=1}^n c_k v_k^{\lambda},\ n\geq 1.
\end{equation}
Then we have $r_n(x_k) = 0$ for $1\leq k\leq n$ and 
\begin{equation}\label{eq:reg_int_res}
s^{\lambda}_n(f) = \sum_{k=1}^n (r_{k-1}, v_k^{\lambda})_{\calhl} v_k. 
\end{equation}
\end{proposition}
\begin{proof}
Since $\{v_k^{\lambda}\}_k$ is $\calh^{\lambda}$-orthonormal, we obtain
\begin{align*}
(r_{k-1}, v_k^{\lambda})_{\calh^{\lambda}} &= \left(f - \sum_{j=1}^{k-1} c_j v_j^{\lambda}, v_k^{\lambda}\right)_{\calh^{\lambda}}= 
\left(f,v_k^{\lambda}\right)_{\calh^{\lambda}},
\end{align*}
and thus \cref{eq:reg_int_res} equals \cref{eq:reg_int1}.
Moreover, by the form \cref{eq:interp_in_calhl} of the interpolant $I_n^{\lambda}(f)$ we obtain
\begin{align*}
r_n(x_k) &= f(x_k) - \sum_{k=1}^n c_k v_k^{\lambda}(x_k) = f(x_k) - \sum_{k=1}^n (f, v_k^{\lambda})_{\calh^{\lambda}} v_k^{\lambda}(x_k)\\
&= f(x_k) - 
I_n^{\lambda}(f)(x_k) = 0.
\end{align*}
\end{proof}

\section{Approximation schemes and error estimation}\label{sec:error}

We recall that a standard way to measure the pointwise interpolation error is via the {power function}. Although we do not consider the case here, we remark that it 
would be possible to have a proper definition also for positive definite kernels. Indeed, it would be sufficient to solve the linear system defining the interpolant 
using the pseudo-inverse of the kernel matrix.

Instead in the following, whenever we mention $s_n^0$ or its power function $P_n$ we implicitly assume that $K$ is strictly positive definite so that both 
objects are well defined, which is instead always the case for the interpolant $I_n^{\lambda}$ if $\lambda>0$. 

For the interpolants $s_n^0$ and $I_n^{\lambda}$ the power function can be defined and computed as 
\begin{align}\label{eq:power_functions}
P_n(x)&:=\sup\limits_{\stackrel{f\in\calh}{f\neq 0}}\frac{\left|f(x) - s_n^0(f)(x)\right|}{\|f\|_{\calh}}= \left\|K(\cdot, x) - s_n^0(K(\cdot, x))\right\|_{\calh}\\
\nonumber Q_n^{\lambda}(x)&:=\sup\limits_{\stackrel{f\in\calhl}{f\neq 0}}\frac{\left|f(x) - I_n^{\lambda}(f)(x)\right|}{\|f\|_{\calhl}}= \left\|K_{\lambda}(\cdot, x) - 
I_n^{\lambda}(K_{\lambda}(\cdot, x))\right\|_{\calhl},
\end{align}
and in both cases from the definition we obtain pointwise error bounds of the form
\begin{align}\label{eq:power_bound}
|f(x) - s_n^0(f)(x)|&\leq P_n(x) \|f\|_{\calh},\;\; x\in\Omega,\;f\in \calh\\
|f(x) - I_n^{\lamdba}(f)(x)|&\leq Q_n^{\lambda}(x) \|f\|_{\calhl},\;\; x\in\Omega,\;f\in \calhl \nonumber,
\end{align}
where the bounds can not be improved for a fixed $x\in\Omega$, if they have to hold for all $f\in\calh$ or $f\in\calhl$. 

To obtain the same kind of error bound as in \cref{eq:power_bound}, we define the power function of regularized interpolation as
\begin{equation}\label{eq:Qn}
P_n^{\lambda}(x):=\sup\limits_{\stackrel{f\in\calh}{f\neq 0}}\frac{\left|f(x) - 
s_n^{\lambda}(f)(x)\right|}{\|f\|_{\calh}},
\end{equation}
which immediately gives
\begin{align*}
|f(x) - s_n^{\lambda}(f)(x)|&\leq P_n^{\lambda}(x) \|f\|_{\calh},\;\; x\in\Omega,\;f\in \calh.
\end{align*}
By this definition it holds indeed $P_n^{\lambda} = P_n$ if $\lambda = 0$, and we have the following result.

\begin{proposition}\label{prop:Qn}
For all $x\in\Omega$ we have
\begin{align}\label{eq:reg_power}
P_n^{\lambda}(x)^2 &= \left\|K(\cdot, x) - s_n^{\lambda}(K(\cdot, x))\right\|_{\calh}^2\\
\nonumber&= K(x, x) - 2 \sum_{k=1}^n v_k(x)^2 +  \sum_{l, k=1}^n v_k(x)v_l(x)(v_k, v_l)_{\calh}. 
\end{align}
\end{proposition}
\begin{proof}
First observe that, using \cref{eq:reg_int2}, for all $f, g\in\calh$ it holds
$$
\left(f, s_n^{\lamdba}(g)\right)_{\calh} = \left(s_n^{\lamdba}(f), g\right)_{\calh}.
$$
To simplify the notation we define $v_x:=K(\cdot, x)$. For any $x \in \Omega$ and $f\in\calh$  we have
\begin{align*}
\left|\left(f - s_n^{\lambda}(f) \right)(x) \right| &= \left|(v_x, f)_{\calh} - (v_x, s_n^{\lambda}(f))_{\calh}\right|=\left|(v_x, f)_{\calh} - 
(s_n^{\lambda}(v_x), f)_{\calh}\right|\\
& = \left|(v_x - s_n^{\lambda}(v_x), f)_{\calh}\right| \leq \left\|v_x - s_n^{\lambda}(v_x)\right\|_{\calh}  \|f\|_{\calh}, 
\end{align*}
so from \cref{eq:Qn} it follows that $P_n^{\lambda }(x) \leq \left\|v_x - s_n^{\lambda}(v_x)\right\|_{\calh}$.

The equality is reached by taking $f:=f_{x} : = v_x - s_n^{\lambda}(v_x)$. Indeed, the norm of $f_x$ is
\begin{align*}
\|f_x\|_{\calh}^2 &= \left\|v_x - s_n^{\lambda}(v_x)\right\|_{\calh}^2 = (v_x, v_x)_{\calh} - 2 (v_x, s_n^{\lambda}(v_x))_{\calh} + (s_n^{\lambda}(v_x), 
s_n^{\lambda}(v_x))_{\calh} \\
& = (v_x, v_x)_{\calh} - 2 (v_x, s_n^{\lambda}(v_x))_{\calh} + (v_x,s_n^{\lambda}(s_n^{\lambda}(v_x)))_{\calh}\\
& = v_x(x) - 2 s_n^{\lambda}(v_x)(x) + s_n^{\lambda}(s_n^{\lambda}(v_x))(x), 
\end{align*}
and by linearity of $s_n^{\lambda}$ we obtain
\begin{align*}
(f_x - s_n^{\lambda}(f_x))(x) &= v_x(x) - s_n^{\lambda}(v_x)(x) - s_n^{\lambda}(v_x)(x) + s_n^{\lambda}(s_n^{\lambda}(v_x))(x) \\
&= v_x(x) - 2 s_n^{\lambda}(v_x)(x)  +  s_n^{\lambda}(s_n^{\lambda}(v_x))(x) =\|f_x\|_{\calh}^2,
\end{align*}
thus ${|f_x(x) - s_n^{\lambda}(f_x)(x)|}/{\|f_x\|_{\calh}} = \|f_x\|_{\calh}=\left\|v_x - s_n^{\lambda}(v_x)\right\|_{\calh}.$

Using the form \cref{eq:reg_int2} of the regularized interpolant, the second equality easily follows. Observe that the terms can not be simplified for $\lambda>0$ 
because the basis 
$\{v_k\}_{k=1}^n$ is not orthogonal in $\calh$.
\end{proof}

Thanks to \cref{eq:power_bound}, upper bounds on the power function give upper 
bounds on the pointwise approximation error achieved by the corresponding approximation scheme. 
The $P$-greedy \cite{DeMarchi2005}  variant of VKOGA uses precisely this idea and selects at each iteration the new point $x_n$ that maximizes $P_n(x)$ over 
$\Omega_h\setminus X_{n-1}$. It can be proven that this selection strategy produces approximants that have a quasi-optimal convergence rate, i.e., $n$ greedily 
selected points  provide, up to a different constant, the same convergence order of $n$ optimally placed ones (see \cite{SH16b}). We would like to achieve the 
same result here by defining  
a suitable $P$-greedy selection rule and prove optimality of the corresponding interpolant. In the case of interpolation, the actual use of this selection 
criterion is possible because also the power function has an efficient 
update rule. For example, in the case of $Q_n^{\lambda}$, for the Newton basis $\{v_k^{\lambda}\}_{k=1}^n$ it holds
\begin{equation}\label{eq:power_update}
Q_n^{\lambda}(x)^2=K_{\lambda}(x, x) - \sum_{k=1}^n v_k^{\lambda}(x)^2 =  Q_{n-1}^{\lambda}(x)^2  - v_n^{\lambda}(x)^2,
\end{equation}
and similarly for $P_n$ using the corresponding Newton basis. Moreover, the convergence results of \cite{SH16b} are possible because of the use of the general results of 
\cite{DeVore2013}, which apply to the case of approximation schemes which are best approximations, like it is the case for interpolation in $\calh$, $\calhl$. 
Instead, it is clear from \cref{prop:Qn} that both these properties are not realized by $P_n^{\lambda}$. Nevertheless, we will overcome the 
problem by relating $P_n^{\lambda}$ to $P_n$ and $Q_n^{\lambda}$ as follows.

We remark that a step of the following proof requires the use of equation \cref{eq:dominik}, which in turn follows from 
\cref{prop:power_and_lagrange}. Both results are proven independently from the next proposition, so we postpone them to simplify the exposition of the 
results.

\begin{proposition}\label{prop:pf_relation}
For $x\in\Omega$ it holds $P_n^{\lamdba}(x)\leq \sqrt{\lambda}$, while  we have
\begin{align}\label{eq:pf_rel_2}
P_n(x) \leq P_n^{\lambda}(x) &\leq  Q_n^{\lambda}(x) \;\;\fa\;\; x\in\Omega\setminus X_n.
\end{align}
In particular 
\begin{align}\label{eq:pf_rel_3}
\left\|P_n^{\lambda}\right\|_{L_{\infty}(\Omega)} &\leq  \left\|Q_n^{\lambda}\right\|_{L_{\infty}(\Omega)}.
\end{align}
\end{proposition}
\begin{proof}
First, for $x\in X_n$ it holds  $|f(x_i) - s_n^{\lambda}(x_i)|\leq \sqrt{\lambda}\|f\|_{\calh}$ (see e.g. Proposition 3.1 in \cite{WendlandRieger2005}), so 
$P_n^{\lamdba}(x)\leq \sqrt{\lambda} $ thanks to the definition \cref{eq:Qn} .
 
The first inequality in \cref{eq:pf_rel_2} follows from \cref{prop:Qn} and the definitions \cref{eq:power_functions}. Indeed, both $s_n^{0}$ and 
$s_n^{\lambda}$ 
are maps into $V(X_n)$, but the interpolant is the best approximation operator in $\calh$, so 
$$
\left\|K(\cdot, x) - s_n^{0}(K(\cdot, x))\right\|_{\calh}\leq \left\|K(\cdot, x) - s_n^{\lambda}(K(\cdot, x))\right\|_{\calh}.
$$
 
For the second inequality, and for $x\notin X_n$ and $f\in\calhl$, we observed in \cref{sec:iterative} that $s_n^{\lambda}(f)(x) = 
I_n^{\lambda}(f)(x)$. Moreover, from Proposition 
\cref{prop:hl}, 
we have that the unit ball in $\calh$ is contained in the unit ball of $\calhl$, thus using 
a standard argument
we can conclude that for  $x\notin X_n$
\begin{align*}
 P^{\lambda}_n(x) &= \sup\limits_{\stackrel{f\in\calh}{f\neq 0}}\frac{\left|f(x) - s_n^{\lambda}(f)(x)\right|}{\|f\|_{\calh}} \leq 
\sup\limits_{\stackrel{f\in\calhl}{f\neq 0}}\frac{\left|f(x) - s_n^{\lambda}(f)(x)\right|}{\|f\|_{\calhl}}\\
&= \sup\limits_{\stackrel{f\in\calhl}{f\neq 0}}\frac{\left|f(x) - I_n^{\lambda}(f)(x)\right|}{\|f\|_{\calhl}}=Q_n^{\lambda}(x).
\end{align*}

Finally, since $Q_n^{\lambda}$ is the power function of the interpolation with the strictly positive definite kernel $K_{\lambda}$, it holds $Q_n^{\lamdba}(x) 
= 0$ if and only if $x\in X_n$. In particular for all $n$ the maximum of $Q_n^{\lambda}$ is reached for $x\notin X_n$, and for this point it holds 
$P_n^{\lambda}(x) \leq Q_n^{\lambda}(x)$. Since $P_n^{\lambda}(x)\leq \sqrt{\lambda}$ as just proved, we just need to show that $Q_n^{\lambda}(x)\geq 
\sqrt{\lambda}$ for $x\in\Omega\setminus X_n$, which follows from \cref{eq:dominik}. 
\end{proof}

An illustration of the relation between the three power functions is provided in Figure \cref{fig:pf_example} in the case $\Omega:=[0,1]$, $X_4:=\{0.1,0.4,0.7,0.8\}$, 
the Gaussian kernel $K(x, y):=\exp\left(- (4 \|x-y\|_2)^2\right)$ and $\lambda = 0.1$.

\begin{figure}[ht!]
\begin{center}
 \includegraphics[width=0.8\textwidth]{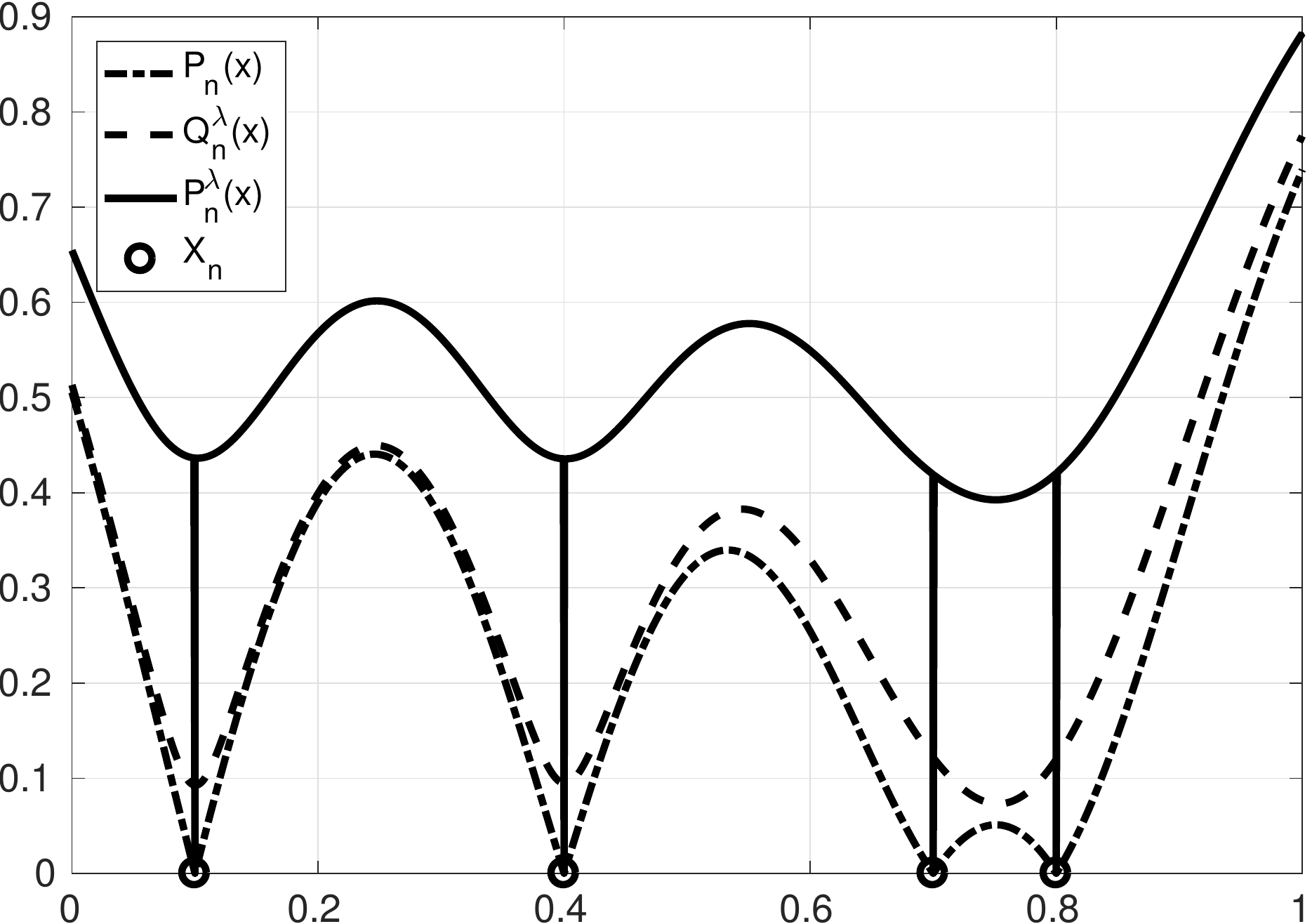}
\end{center}
\caption{Power functions $P_n$ (interpolation in $\calh$), $P_n^{\lambda}$ (regularized interpolation) and the discontinuous $Q_n^{\lambda}$ (interpolation in 
$\calhl$) for the Gaussian kernel in $[0, 1]$ and a given point set $X_n$, $n=4$.}\label{fig:pf_example}
\end{figure}

Using  $Q_n^{\lambda}$ as an upper bound for $P_n^{\lambda}$ solves both issues, since $Q_n^{\lambda}$ can be efficiently updated by \cref{eq:Qn}, and it is 
related to a best 
approximation 
operator, i.e., the orthogonal projection in $\calhl$. 
To complete the analysis we need to estimate the decay rate of $Q_n^{\lambda}$, an we will do so by relating it to the one of $P_n$. 
To this end, we first state the following, which is an easy generalization of the case of interpolation.

\begin{proposition}\label{prop:lagrange}
Let $X_n\subset\Omega$ and $\lambda>0$ if $K$ is positive definite or $\lambda\geq 0$ if $K$ is strictly positive definite. 
Then there exists a Lagrange basis $\{\ell_j^{\lambda}\}_{j=1}^n$ of $V(X_n)$ s.t.
\begin{equation}\label{eq:lagrange_interp}
s_n^{\lambda}(f)(x) =\sum_{i=1}^n f(x_i)\ell_j^{\lambda}(x), \; x\in\Omega .
\end{equation}
The basis is defined by 
\begin{equation}\label{eq:lagrange}
\ell_j^{\lambda}=\sum_{i=1}^n ((A+{\lambda} I)^{-1})_{ij} K(\cdot, x_i).
\end{equation}
and $\ell_j^{\lambda}(x_i) = \delta_{ij}$ if $\lambda=0$.

Moreover, the $\ell_2$-Lebesgue function $\Lambda_{n,2}^{\lambda}(x)$ can be computed for all $x\in\Omega$ as 
\begin{equation}
\Lambda_{n, 2}^{\lambda}(x):=\sup\limits_{\stackrel{f\in\ns}{f\neq 0}} 
\frac{|s_n^{{\lambda}}(f)(x)|}{\|f_{|X_n}\|_{\ell_2(X_n)}}=\sqrt{\sum_{j=1}^n\ell_j^{\lambda}(x)^2},
\end{equation}
and it holds $\Lambda_n^{\lambda}(x)<\Lambda_n^{\mu}(x)$ if $\lambda>\mu\geq 0$.
\end{proposition}
\begin{proof}
It is clear that \cref{eq:lagrange} defines a basis of $V(X_n)$ since the coefficient matrix is invertible, and formula \cref{eq:lagrange_interp} holds. 
In particular, it holds
\begin{equation*}
|s_n^{\lambda}(f)(x)| \leq \sum_{j=1}^n \left|f(x_j) \ell_j^{\lambda}(x)\right|\leq \left(\sum_{j=1}^n f(x_j)^2 \right)^{1/2} 
\left(\sum_{j=1}^n\ell_j^{\lambda}(x)^2\right)^{1/2},
\end{equation*}
and the equality is reached, for a fixed $x\in\Omega$, by considering $f:=f_x$ with $f_x(x_j) := \ell_j^{\lambda}(x)$.

Defining $k_x:=[K(x, x_1), \dots, K(x, x_n)]^T$, for $x\in\Omega$ we have from \cref{eq:lagrange} that $\ell_j^{\lamdba}(x) = ((A+{\lambda} I)^{-1} k_x)_j$,  thus 
\begin{align}\label{eq:intermediate_lagrange}
\Lambda_{n, 2}^{\lambda}(x)^2 &= \sum_{j=1}^n\ell_j^{\lambda}(x)^2  = k_x^T (A+{\lambda} I)^{-2} k_x.
\end{align}
In particular for $\lambda>\mu\geq 0$ we have 
\begin{align*}
\Lambda_n^{\mu}(x) - \Lambda_n^{\lambda}(x) = k_x^T ((A+{\mu} I)^{-2} - (A+{\lambda} I)^{-2})k_x \geq 0,
\end{align*}
since the matrix is positive semidefinite. Indeed, if $A= U\Sigma U^T$ is an eigen-decomposition of $A$ with 
$\Sigma:=\mbox{diag}\{\sigma_i\}$ and $\sigma_1\geq \sigma_2\geq \sigma_n\geq 0$, we have $A +\lambda I = U (\Sigma + \lambda I) U^T$ and thus $U ((A+{\mu} I)^{-2} - 
(A+{\lambda} I)^{-2}) U^T$ is a diagonal matrix with diagonal elements
\begin{align*}
\frac{1}{(\sigma_i + \mu)^{2}}  - \frac{1}{(\sigma_i + \lambda)^{2}},
\end{align*}
which are non negative if $0\leq \mu\leq \lambda$.

\end{proof}

Using $\Lambda_{n, 2}^{\lambda}$ we can now exactly quantify the difference between $P_n^{\lambda}$ and $Q_n^{\lambda}$.

\begin{proposition}\label{prop:power_and_lagrange}
For all $x\in\Omega$ and $0\leq \mu\leq \lambda$ we have
\begin{align}\label{eq:second_bound_qn}
Q_n^{\lambda}(x)^2 \leq P_n^{\mu}(x)^2 + \lambda\left(1 + \Lambda_{n, 2}^{\mu}(x)^2\right),
\end{align}
and equality holds for $\mu=\lambda$ and $x\notin X_n$.
\end{proposition}
\begin{proof}
We use the formula \cref{eq:reg_power} for the power function $P_n^{\mu}$, but we express the interpolant in terms of the Lagrange basis as in 
\cref{eq:lagrange_interp}. We use again the same notation $k_x$ as in the previous proof and we obtain
\begin{align}\label{eq:power_intermediate}
P_n^{\mu}(x)^2 &= \left\|K(\cdot, x) - s_n^{\mu}(K(\cdot, x))\right\|_{\calh}^2\\
\nonumber&=K(x, x) - 2 \sum_{j=1}^n \ell_j^{\mu}(x) K(x, x_j) + \sum_{i,j=1}^n \ell_j^{\mu}(x) 
\ell_i^{\mu}(x) K(x_i, x_j)\\
\nonumber&= K(x, x) - 2 k_x ^T (A + \mu I)^{-1} k_x + k_x ^T (A + \mu I)^{-1} A (A + \mu I)^{-1} k_x\\
\nonumber&= K(x, x) - k_x ^T (A + \mu I)^{-1}( 2(A + \mu I)- A) (A + \mu I)^{-1} k_x\\
\nonumber& = K(x, x) - k_x ^T (A + \mu I)^{-1}(A  + 2\mu I) (A + \mu I)^{-1} k_x.
\end{align}
In particular, for $\mu = 0$ we obtain the usual formula
\begin{align*}
P_n(x)^2 = P_n^{0}(x)^2 &= K(x, x) - k_x ^T A^{-1}A  A ^{-1} k_x = K(x, x) - k_x ^T A^{-1} k_x,
\end{align*}
and the same holds for the interpolatory power function $Q_n^{\lambda}$, where instead we define $k_x^{\lambda}:=[K_{\lambda}(x, x_1), \dots, K_{\lambda}(x, 
x_n)]^T$ and obtain
\begin{align}\label{eq:qower_intermediate}
Q_n^{\lambda}(x)^2 &= K_{\lambda}(x, x) - (k_x^{\lambda}) ^T (A + \lambda I)^{-1} k_x^{\lambda}.
\end{align}
Moreover, we have $\Lambda_n^{\mu}(x)^2 = k_x^T (A+\mu I)^{-2} k_x$ from \cref{eq:intermediate_lagrange}.

If $x\in X_n$ it holds $Q_n^{\lambda}(x) = 0$, so \cref{eq:second_bound_qn} easily follows since the right hand side is non negative.

If instead $x\notin X_n$, it holds $k_x^{\lambda} = k_x$ since  $K_{\lambda}(x, x_i) = K(x, x_i)$ for $1\leq i \leq n$, thus \cref{eq:power_intermediate} and 
\cref{eq:qower_intermediate} imply that
\begin{align*}
&Q_n^{\lambda}(x)^2 - P_n^{\mu}(x)^2 - \lambda \Lambda_n^{\mu}(x)^2 - \lambda=\\
&= k_x^T \left(-(A + \lambda I)^{-1} + (A + \mu I)^{-1}(A  + 2\mu I) (A + \mu 
I)^{-1} - \lambda (A+\mu I)^{-2}\right) k_x.
\end{align*}
We denote as $B(\mu,\lambda)$ the matrix in the right hand side. Using matrices $U$, $\Sigma$ as in the proof of \cref{prop:lagrange}, we have that $U B(\mu, 
\lambda) U^T$ has diagonal elements
\begin{align*}
\rho_i:=&-\frac{1}{\lambda+\sigma_i}+\frac{2 \mu +\sigma_i}{(\mu +\sigma_i)^2}-\frac{\lambda }{(\mu +\sigma_i)^2}=-\frac{(\lambda -\mu )^2}{(\lambda +\sigma_i) (\mu 
+\sigma_i)^2},
\end{align*}
which are negative for all $0\leq \mu<\lambda$, and exactly zero for $\mu=\lambda$, i.e., $B(\mu, \lambda)$ is negative definite for $\mu<\lambda$ and 
the zero matrix for $\mu=\lamdba$, and thus the statements follows.
\end{proof}

In proving the convergence of the algorithm we will only need the case $K$ strictly positive definite and $\mu=0$ in \cref{eq:pf_rel_2}. Nevertheless, the 
case $\mu= \lambda>0$ allows to conclude that the right hand side is well defined also when $K$ is positive definite, since $P_n^{\lambda}$ is still well defined in 
this 
case, and especially in this case for $x\notin X_n$ \cref{eq:pf_rel_2} implies that 
\begin{equation}\label{eq:dominik}
Q_n^{\lambda}(x)^2 = P_n^{\lambda}(x)^2 + \lambda\left(1 + \Lambda_{n, 2}^{\lambda}(x)^2\right) \geq \lambda,
\end{equation}

i.e., we can expect $\left\|Q_n^{\lambda}\right\|_{L_{\infty}(\Omega)}$ to converge at most to $\sqrt{\lambda}$ for $n\to\infty$, and not to $0$.

\subsection{The case of translational invariant kernels}\label{sec:tr_inv}
In some notable cases, convergence rates can be derived for the regularized interpolation process using {sampling inequalities}. They apply to translational 
invariant kernels $K(x, y) :=\Phi(x -y)$ which are strictly positive definite on ${\mathbb R}^d$, and such that $\Phi$ has a continuous Fourier transform $\hat\Phi$ on 
${\mathbb R}^d$.
In this case, the native space on a set $\Omega\subset{\mathbb R}^d$ satisfying an interior cone condition can be described in terms of $\hat \Phi$. 
In particular, if there exist $c_{\Phi}, C_{\Phi}>0$ and $\tau\in\N$, $\tau>d/2$, such that
$$
c_{\Phi}\left(1 + \|\omega\|_2^2\right)^{-\tau} \leq \hat\Phi(\omega) \leq C_{\Phi}\left(1 + \|\omega\|_2^2\right)^{-\tau},
$$ 
shortly $\hat \Phi(\omega) \sim (1 + \|\omega\|_2^2)^{-\tau}$, then $K\in\mathcal C^{2\tau}({\mathbb R}^d\times{\mathbb R}^d)$ and $\calh({\mathbb R}^d)$ is norm 
equivalent to the 
Sobolev space $W_2^{\tau}({\mathbb R}^d)$. Examples of kernels of this type are e.g. the Wendland kernels \cite{Wendland1995a}.
If instead the Fourier transform decays faster than any polynomial one has a native space of infinitely smooth functions, and this is the case e.g. of the Gaussian 
kernel or the inverse multiquadric (IMQ) kernel.

Sampling inequalities  quantify the error in terms of the {fill distance}
$$
h_n:=h_{X_n, \Omega}:= \sup_{x\in\Omega}\min_{x_j \in X_n} \|x-x_j\|_2,
$$
and they bound the error in approximating the derivative $D^a(f)$, 
where $a:=(a_1, \dots, a_d)\in\N_0^d$ is a multi index and $|a|:=a_1+\dots+a_d$.

We state only the particular case of the bounds in the $L_{\infty}$-norm, and refer to the cited papers for a more general version.
\begin{theorem}[\cite{Rieger2008b,WendlandRieger2005}]\label{th:sampling}
Assume $\Omega$ is bounded and satisfies an interior cone condition.
\begin{enumerate}[i)]
 \item\label{case1} If $\hat \Phi(\omega) \sim (1 + \|\omega\|_2^2)^{-\tau}$, $\tau>d/2$, there exist 
constants $C, h_0$ such that for all $X_n\subset\Omega$ with $h_n\leq h_0$, $0\leq |a|<\tau-d/2$, and $f\in\calh$ it holds
\begin{equation}\label{eq:sampling_finite_pre}
\left\|D^a(f)\right\|_{L_{\infty}(\Omega)} \leq C h_n^{-|a|} \left(h_n^{\tau - d/2} \|f\|_{\calh} + \|f_{|X_n}\|_{\ell_{\infty}(X_n)}\right),
\end{equation}
and, under the same hypotheses and  for any $\lambda>0$, it holds
\begin{equation}\label{eq:sampling_finite}
\left\|D^a\left(f - s_n^{\lambda}(f)\right)\right\|_{L_{\infty}(\Omega)} \leq C h_n^{-|a|} \left(h_n^{\tau - d/2} + \sqrt{\lambda}\right) \|f\|_{\calh}.
\end{equation}

 \item\label{case2} Assume additionally that $\Omega$ has a Lipschitz boundary. If $K$ is the Gaussian or IMQ kernel there exist constants $C', C'', h_0'$ such that for 
all 
$X_n\subset\Omega$ with $h_n\leq h_0'$, $a\in\N_0^d$, and $f\in\calh$ it holds
\begin{equation}\label{eq:sampling_infinite_pre}
\left\|D^a\left(f)\right)\right\|_{L_{\infty}(\Omega)} \leq e^{-C' / \sqrt{h_n}}\|f\|_{\calh} + C'' h_n^{-|a|} \|f_{|X_n}\|_{\ell_{\infty}(X_n)},
\end{equation}
and, under the same hypotheses and  for any $\lambda>0$, it holds
\begin{equation}\label{eq:sampling_infinite}
\left\|D^a\left(f - s_n^{\lambda}(f)\right)\right\|_{L_{\infty}(\Omega)} \leq \left(2 e^{-C' / \sqrt{h_n}}+C'' \sqrt{\lambda}\ h_n^{-|a|}\right) \|f\|_{\calh}.
\end{equation}
\end{enumerate}
\end{theorem}

Observe that in the case of the Gaussian the exponential term in \cref{eq:sampling_infinite_pre} and \cref{eq:sampling_infinite} can be improved to $e^{C' \log(h_n) / 
\sqrt{h_n}}$, with a different 
constant $C'$ (\cite[Theorem 3.5]{Rieger2008b}). Moreover, the constant $C$ in \cref{eq:sampling_finite_pre} and \cref{eq:sampling_finite} includes a factor 
depending on $c_{\Phi}, C_{\Phi}$, which is needed to express the inequalities in terms of the $\calh$-norm instead of the Sobolev norm.

Both bounds are used in the corresponding papers to conclude, among other findings, that there is an upper bound on the maximal $\lambda$ to be used. 
Indeed, the resulting convergence rates are optimal in the sense that, by choosing $\lambda\leq h_n^{2 \tau -d}$ in the first case or $\lambda \leq (C'')^{-2} 
\exp\left(2 C' / \sqrt{h_n}\right) h_n ^{2|a|}$ in the second one, one gets, up to constants, the same order of pure interpolation (see 
\cite{Schaback1995}), while solving a potentially much better conditioned linear system. We use these bound to deduce convergence rates of $Q_n$ and 
$P_n^{\lambda}$.

To quantify the decay rate of $Q_n^{\lambda}$ using \cref{prop:power_and_lagrange}, we also need to control $\Lambda_{n,2}^{\lambda}$. This kind of stability 
is usually related to the \textit{separation distance}
$$
q_n:=q_{X_n}:= \frac12 \min_{x_i\neq x_j \in X_n} \|x_i-x_j\|_2,
$$
which can be used to estimate a lower bound on the minimal eigenvalue of the kernel matrix. It is known from \cite{DeMarchiSchaback2008,DeMarchi2010} that in the case of 
\cref{case1} of \cref{th:sampling} there is a constant $c>0$ such that
\begin{equation}\label{eq:dem_sch_bpund}
\left\|\Lambda_{n, 2}^{0}\right\|_{L_{\infty}(\Omega)} \leq c \left( \left(\frac{h_n}{q_n}\right)^{\tau-d/2} + 1\right),
\end{equation}
and a bound on $\Lambda_{n, 2}^{0}$ would be sufficient in view of \cref{prop:power_and_lagrange}. Nevertheless, the same is not true for infinitely smooth 
kernels, since in this case the lower bound on the smallest eigenvalue and the upper bound on the error have a significant gap (see e.g. 
\cite{Diederichs2017}). Namely, $q_n$ and $h_n$ appear in the right hand side with different exponents, so the upper bound in the last equation is not bounded 
even for $h_n\asymp q_n$. Instead, we can employ the same technique of \cite{DeMarchiSchaback2008} to obtain a similar result in 
the case of regularized interpolation.

\begin{proposition}\label{prop:stability}
Under the same assumptions of the two cases of \cref{th:sampling}, and with the same constants $C$, $C'$, $C''$, we have the following:
\begin{enumerate}[i)]
 \item For finitely smooth kernels it holds
\begin{equation}
\left\|\Lambda_{n, 2}^{\lambda}\right\|_{L_{\infty}(\Omega)} \leq C \left(\frac{h_n^{\tau-d/2}}{\sqrt{\lambda}} + 2\right).
\end{equation}
\item For infinitely smooth kernels it holds
\begin{equation}
\left\|\Lambda_{n, 2}^{\lambda}\right\|_{L_{\infty}(\Omega)} \leq\frac{ 2 e^{-C' / \sqrt{h_n}}}{\sqrt{\lambda}}  + 2 C'' .
\end{equation}
\end{enumerate}
\end{proposition}
\begin{proof}
The regularized interpolant $s_n^{\lambda}(f)$ is a minimizer of $J(s):=\sum_{j=1}^n(s(x_j) -f(x_j))^2 + \lambda \|s\|_{\calh}^2$, so in particular we have
\begin{align*}
\|s_n^{\lambda}(f) - f\|_{\ell_2(X_n)}^2 + \lambda \|s_n^{\lambda}(f)\|_{\calh}^2 = J(s_n^{\lambda}(f))\leq J(0) = \|f\|_{\ell_2(X_n)}^2.
\end{align*}
It follows that $\|s_n^{\lambda}(f)\|_{\calh} \leq \frac1{\sqrt{\lambda}}  \|f\|_{\ell_2(X_n)}$ and, by the triangle inequality, 
\begin{align*}
 \|s_n^{\lambda}(f)\|_{\ell_2(X_n)} \leq  \|s_n^{\lambda}(f) - f\|_{\ell_2(X_n)} +  \|f\|_{\ell_2(X_n)} \leq 2 \|f\|_{\ell_2(X_n)}.
\end{align*}
We can use these two bounds in the sampling inequalities \cref{eq:sampling_finite_pre}, \cref{eq:sampling_infinite_pre} and the fact that 
$\|f\|_{\ell_{\infty}(X_n)}\leq \|f\|_{\ell_{2}(X_n)}$ for any function. In the first case we obtain
\begin{align*}
\|s_n^{\lambda}(f)\|_{L_{\infty}(\Omega)} &\leq C \left( h_n^{\tau-d/2} \|s_n^{\lambda}(f)\|_{\calh} + \|s_n^{\lambda}(f)\|_{\ell_{\infty}(X_n)}\right)\\
&\leq C \left(\frac{h_n^{\tau-d/2}}{\sqrt{\lambda}}  + 2 \right) \|f\|_{\ell_2(X_n)},
\end{align*}
and with the second one to obtain
\begin{align*}
\|s_n^{\lambda}(f)\|_{L_{\infty}(\Omega)} &\leq 2 e^{-C' / \sqrt{h_n}} \|s_n^{\lambda}(f)\|_{\calh} + C'' \|s_n^{\lambda}(f)\|_{\ell_{\infty}(X_n)}\\
&\leq \left(\frac{ 2 e^{-C' / \sqrt{h_n}}}{\sqrt{\lambda}}  + 2 C'' \right) \|f\|_{\ell_2(X_n)}.
\end{align*}
These two bounds give the result using the definition of $\Lambda_{n, 2}^{\lambda}$.
\end{proof}
\begin{remark}
We remark that these bounds do not depend on the separation distance, and the first one provides asymptotically better bounds than \cref{eq:dem_sch_bpund}.

Moreover, although not used in this paper, we remark that the same argument of the last proof, and the fact that $\|f\|_{\ell_{2}(X_n)}\leq 
\sqrt{n}\|f\|_{\ell_{\infty}(X_n)}$, allow to conclude that the standard 
$\ell_{\infty}$ Lebesgue constant of regularized interpolation satisfies
$$
\left\|\Lambda_{n, \infty}^{\lambda}\right\|_{L_{\infty}(\Omega)} \leq C \left(\sqrt{n}\ \frac{ h_n^{\tau-d/2}}{\sqrt{\lambda}} + 
2\right)
$$
in the first case, and similarly in the second one. In particular, this means that the Lebesgue constant is asymptotically bounded for all points $X_n$ 
such that 
$h_n^{\tau-d/2}\leq c\ n^{-1/2}$.
 
\end{remark}

Combining \cref{th:sampling} and \cref{prop:stability} we have the following.
\begin{proposition}\label{prop:decay_power}
In the setting of \cref{th:sampling}, we have the following cases:
\begin{enumerate}[i)]
 \item If $h_n\leq h_0$  it holds 
\begin{align*}
\left\|P_n^{\lambda}\right\|_{L_{\infty}(\Omega)} &\leq C\left( h_n^{\tau-d/2} +  \sqrt{\lambda}\right)\\
\;\;\left\|Q_n^{\lambda}\right\|_{L_{\infty}(\Omega)} &\leq 2 C  h_n^{\tau-d/2} +  \left(3 C  + 1\right) \sqrt{\lambda}.
\end{align*}
\item If $h_n\leq h_0'$  it holds
\begin{align*}
\left\|P_n^{\lambda}\right\|_{L_{\infty}(\Omega)} &\leq 2 e^{-C' / \sqrt{h_n}}+C'' \sqrt{\lambda}\\
\;\;\left\|Q_n^{\lambda}\right\|_{L_{\infty}(\Omega)} &\leq 4 e^{-C' / \sqrt{h_n}}+ \left(3 C'' + 1 \right) \sqrt{\lambda}.
\end{align*}
\end{enumerate}
\end{proposition}
\begin{proof}
The bounds on $P_n^{\lambda}$ are just the application of the bounds \cref{eq:sampling_finite} and \cref{eq:sampling_infinite} with $a = 0^T$ to the definition 
\cref{eq:Qn} of $P_n^{\lambda}$. Moreover, from \cref{prop:power_and_lagrange} with $\mu=\lambda$ we have
\begin{align*}
\left\|Q_n^{\lambda}(x)\right\|_{L_{\infty}(\Omega)} \leq \left\|P_n^{\lambda}\right\|_{L_{\infty}(\Omega)} + \sqrt{\lambda}\left(1 +  \left\|\Lambda_{n, 
2}^{\lambda}\right\|_{L_{\infty}(\Omega)}\right),
\end{align*}
and we can use the bounds on $\left\|P_n^{\lambda}\right\|_{L_{\infty}(\Omega)}$ and \cref{prop:stability}.

In the first case we obtain

\begin{align*}
\left\|Q_n^{\lambda}(x)\right\|_{L_{\infty}(\Omega)} &\leq C \left(h_n^{\tau - d/2} + \sqrt{\lambda}\right) + 
 \sqrt{\lambda}\left(1 + C \left(\frac{h_n^{\tau-d/2}}{\sqrt{\lambda}} + 2\right)\right)\\
&\leq 2 C  h_n^{\tau-d/2} +  \left(3 C  + 1\right) \sqrt{\lambda},
\end{align*}

while in the second one it holds
\begin{align*}
\left\|Q_n^{\lambda}(x)\right\|_{L_{\infty}(\Omega)} &\leq 2 e^{-C' / \sqrt{h_n}}+C'' \sqrt{\lambda} + \sqrt{\lambda}\left(1 + \frac{ 2 e^{-C' / 
\sqrt{h_n}}}{\sqrt{\lambda}}  + 2 C''\right)\\
&\leq 4 e^{-C' / \sqrt{h_n}}+ \left(3 C'' + 1 \right) \sqrt{\lambda}.
\end{align*}

\end{proof}

\section{Greedy selection rules and convergence}\label{sec:greedy_selections}
Using \cref{prop:residual} and \cref{prop:pf_relation} we can define two selection rules which generalize the $f$- and $P$-greedy selections of 
interpolation as follows: The regularized version of $f$-greedy is defined by selecting the new point $x_n$ as
$$
x_{n}:=\argmax_{x\in \Omega_h} |r_{n-1}(x)|.
$$
This selection can be performed efficiently thanks to the update rule 
\cref{eq:update}, and it holds $r_{n}(x_k) = 0$ for $1\leq k\leq n$, so no point is selected more than once.

The regularized version of $P$-greedy, instead, selects 
$$
x_{n}:=\argmax_{x\in \Omega_h} Q_{n-1}^{\lambda}(x),
$$
which is just the standard $P$-greedy selection, but applied to the kernel $K_{\lambda}$. In particular $Q_n^{\lambda}(x)=0$ for $x\in X_n$, so again no point is 
selected more than once, and the power function can be updated efficiently using \cref{eq:power_update}. Moreover, thanks to Proposition 
\cref{prop:pf_relation} any upper bound on $\|Q_n^{\lambda}\|_{L_{\infty}(\Omega)}$ provides an upper bound on $\|P_n^{\lambda}\|_{L_{\infty}(\Omega)}$, so it makes 
sense 
to select points to minimize $Q_n^{\lambda}$ in order to minimize $P_n^{\lambda}$. 


We remark that both selection strategies are well defined also for $K$ positive definite if $\lambda>0$, and they are nothing but the standard $f$- and 
$P$-greedy selections applied to $K_{\lambda}$. In particular, the selection of the points and the construction of the regularized interpolants can be obtained just by 
running VKOGA with kernel $K_{\lambda}$, and replacing $K_{\lambda}$ with $K$ after the computation to obtain the desired regularized interpolant.


\subsection{Convergence rates for $P$-greedy selection}\label{sec:convergence}
We can now prove rates of convergence for the new $P$-greedy selection rule. In particular, we prove that $n$ points selected by this criterion and $n$ 
optimally chosen points give power functions such that $\left\|Q_n^{\lambda}\right\|_{L_{\infty}(\Omega)}$ decays with the same rate.

The result is obtained by applying the theory of \cite{SH16b}, which holds for the power function of a strictly positive definite kernel. We refer to this 
paper for the details of the proof. The idea is the following: If there exists a certain placement of $n$ points such that the corresponding power function has 
a given decay rate in terms of $n$, then $n$ points selected by the $P$-greedy algorithm give, up to constants, a power function with the same decay.

To apply this result here, we need first a decay rate on $\left\|Q_n^{\lambda}\right\|_{L_{\infty}(\Omega)}$ in terms of $n$, and this is obtained by a 
standard technique. Indeed, since the bounds of both \cref{th:sampling} and \cref{prop:decay_power} hold for any 
$X_n$ provided $h_n$ is small enough, one can choose in particular a 
sequence 
$\{X_n\}_{n\in\N}$ of quasi uniform points, i.e., such that there exists a uniformity constant $\gamma>1$ such that
$$
h_n\leq \gamma q_n,\; n\in \N,
$$
and this can be shown to imply the existence of a constant $C_{\Omega, \gamma}$ such that
$$
h_n \leq C_{\Omega, \gamma} n^{-1/d},\; n\in \N.
$$
Combining this observation with \cref{prop:decay_power} we immediately obtain the following.

\begin{proposition}\label{prop:conv}
Assume the hypotheses of \cref{th:sampling} hold, and let $\{X_n\}_{n\in\N}\subset\Omega$ be a sequence of quasi uniform points with 
uniformity constant $\gamma>1$. Then the following hold. 
\begin{enumerate}[i)]
 \item\label{iiii} For any $n\in\N$ with $C_{\Omega, \gamma} n^{-1/d}\leq h_0$ we have
\begin{align}\label{eq:decay_power_n}
\left\|Q_n^{\lambda}\right\|_{L_{\infty}(\Omega)} 
&\leq C_0  \left(n^{-\tau/d +1/2} +   \sqrt{\lambda}\right),
\end{align}
with $C_0:=\max\left(2 C C_{\Omega, \gamma}^{\tau-d/2}, 3 C  + 1\right)$.
 \item For any $n\in\N$ with $C_{\Omega, \gamma} n^{-1/d}\leq h_0'$ we have
\begin{align}
\left\|Q_n^{\lambda}\right\|_{L_{\infty}(\Omega)} 
&\leq C'_0\left(e^{-c'_0 n^{-1/2d}}+ \sqrt{\lambda}\right),
\end{align}
with $C'_0:=\max\left(4, 3 C'' + 1\right)$, $c'_0:=C'C_{\Omega, \gamma}^{-1/2}$.
\end{enumerate}
\end{proposition}

Second, the fact that a given decay rate is carried over to the decay rate of the greedy selected points is proven in \cite{SH16b} by using the results of 
\cite{DeVore2013}. To do so, we first prove the following slight generalization of \cite[Corollary 3.3]{DeVore2013} in order to deal with the present case, 
where the convergence is to $\sqrt{\lambda}$, and not to zero. The proof is postponed to Appendix 
\cref{appendix} as it is mainly unrelated to the content of this paper.

\begin{proposition}\label{prop:DeVore_generalization}
Let $\calh$ be a Hilbert space, $\mathcal V\subset\calh$ a subset, and 
\begin{align*}
d_n(\mathcal V, \ns) : = \inf_{\stackrel{V_n\subset\ns}{\dim(V_n) = n}} \sup_{f\in \mathcal V} 
\|f - \Pi_{V_n}(f)\| 
\end{align*}
be the Kolmogorov width of $\mathcal V$ in $\calh$. Let $\sigma_n:= \sup_{f\in \mathcal  V} \|f - \Pi_{\bar V_n}(f)\| $, where $\bar V_n$ is selected by the 
greedy algorithm of \cite{DeVore2013}. Then
\begin{enumerate}[i)]
\item\label{item:label1} $\sigma_n\leq \sqrt{2 \sigma_0} \min\limits_{1\leq m<n} d_m^{\frac{n-m}{n}}$ for all $n\in\N$.
\item\label{item:label2} If there are constants $C_0, \eta>0$ such that $d_n\leq C_0 (n^{-\alpha} + \eta)$ for all $n\in\N$, then 
$\sigma_n\leq C_1(n^{-\alpha} + \eta)\; \fa \; n\in\N$, with $C_1:=2^{1+5\alpha} C_0$.
 \item\label{item:label3} If there are constants $c_0, C_0, \eta>0$ such that $d_n\leq C_0 \left( e^{-c_0 n^{-\alpha}} + \eta\right)$ for all $n\in\N$, then 
$\sigma_n\leq C_1 \left( e^{-c_1 n^{-\alpha}} + \eta\right)\; \fa \; n\in\N$, with $C_1:=\sqrt{2 C_0 \sigma_0}$, $c_1:=2^{-1-2\alpha} c_0$.
\end{enumerate}
\end{proposition}

Using \cref{prop:conv}, \cref{prop:DeVore_generalization}, and \cite{SH16b}, we finally obtain the following.

\begin{theorem}\label{th:final}
Assume the hypotheses of \cref{th:sampling} hold, and let $\{X_n\}_{n\in\N}\subset\Omega$ be a sequence of points selected by the regularized 
$P$-greedy algorithm. Then the following hold. 
\begin{enumerate}[i)]
 \item\label{iiii} For any $n\in\N$ with $C_{\Omega, \gamma} n^{-1/d}\leq h_0$ we have
\begin{align}\label{eq:decay_power_n}
\left\|Q_n^{\lambda}\right\|_{L_{\infty}(\Omega)} 
&\leq C_1  \left(n^{-\tau/d +1/2} +   \sqrt{\lambda}\right),
\end{align}
with $C_1:=2^{1+5\alpha} C_0$ and $C_0$ as in \cref{prop:conv}.
 \item For any $n\in\N$ with $C_{\Omega, \gamma} n^{-1/d}\leq h_0'$ we have
\begin{align}
\left\|Q_n^{\lambda}\right\|_{L_{\infty}(\Omega)} 
&\leq C'_0\left(e^{-c'_0 n^{-1/2d}}+ \sqrt{\lambda}\right),
\end{align}
with $C_1:=\sqrt{2 C_0' \sqrt{K_{\lambda}(x,x)}   }$, $c_1:=2^{-1-2\alpha} c_0'$ and $C'_0$, $c'_0$ as in \cref{prop:conv}
\end{enumerate}
\end{theorem}

\section{Experiments}\label{sec:experiments}

We conclude this paper by demonstrating the decay rates of the power function for translational invariant kernels of different smoothness.
We remark that the $f$-greedy variant of the algorithm has been recently used to construct a data-based surrogate from simulation data in \cite{KSHH2017}. We point to 
this paper for a practical application scenario.

In the following, we use the Wendland kernel $W_{4, d}$ (see \cite{Wendland1995a}) and the Gaussian kernel $G(x, y):= \exp(\varepsilon^2 \|x-y\|^2)$, which are 
respectively members of the two classes of kernels considered in \cref{th:final}. In both cases, the shape parameter is fixed to $\varepsilon=1$.  The set 
$\Omega$ is the unit ball in $\R^2$, which is represented by a discretization $\Omega_h$ obtained by restricting a regular grid in $[-1,1]^2$ to $\Omega$, so that the 
number of points is $N\approx 20000$.  Both the greedy selection and the computation of the $L_{\infty}(\Omega)$ norms are performed on this set.

For both kernels, we compare the decay of the standard power function (i.e., $\lambda = 0$) with $Q_n^{\lambda}$ for $\lambda = 10^{-14}, 10^{-12}, \dots, 10^{-6}$.

The  $P$-greedy algorithm is stopped when the maximum of the power function is below the tolerance of $10^{-16}$ (which happens only for $\lambda=0$) or when 
$n=1000$ points are selected.

The results are in \cref{fig:decay} for $W_{4,d}$ (left) and $G$ (right). For each value of $\lambda$, we compute the minimal coefficient $c$ such that 
$$
\left\|Q_n^{\lambda}\right\|_{L_{\infty}(\Omega_h)} \leq c_{\lambda} \left(\|P_n\|_{L_{\infty}(\Omega_h)} + \sqrt{\lambda}\right).
$$
The plots show both the decay of the power 
functions (in solid lines) and the curves $c_{\lambda} \left(\|P_n\|_{L_{\infty}(\Omega_h)} + \sqrt{\lambda}\right)$. Observe that, in the case $\lambda=0$, the 
algorithm stops much earlier, say at $n<1000$, so the dotted curves are limited to the first $n$ iterations. The computed coefficients are in \cref{tab:coeffs}.

\begin{table}
 \begin{tabular}{|l|c|c|c|c|c|}
  \hline
  &$\lambda = 10^{-14}$&  $\lambda = 10^{-12}$&$\lambda = 10^{-10}$&$\lambda = 10^{-8}$&$\lambda = 10^{-6}$\\
\hline
  Wendland & 1.00 &  1.00 & 1.02 &  1.12 &   1.07\\
  \hline
  Gaussian &1.36 & 1.48 & 1.46 & 1.54 & 1.44 \\
  \hline
 \end{tabular}
\caption{Coefficients $c_{\lambda}$ relating the power functions $P_n$ and $Q_n^{\lambda}$ as described in \cref{sec:experiments}.}
\label{tab:coeffs}
\end{table}

\begin{figure}
\begin{tabular}{cc}
\includegraphics[width=0.5\textwidth]{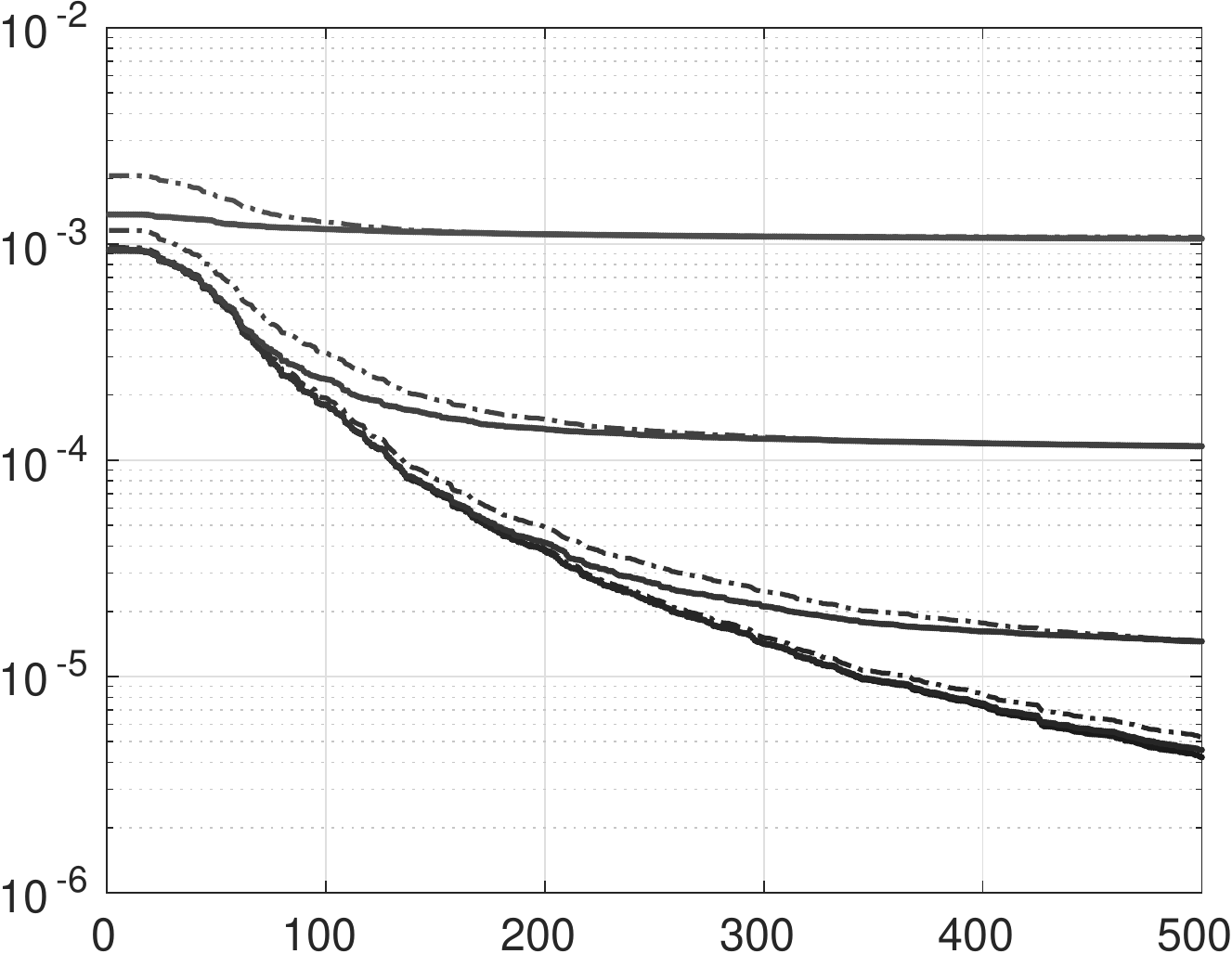}&
\includegraphics[width=0.5\textwidth]{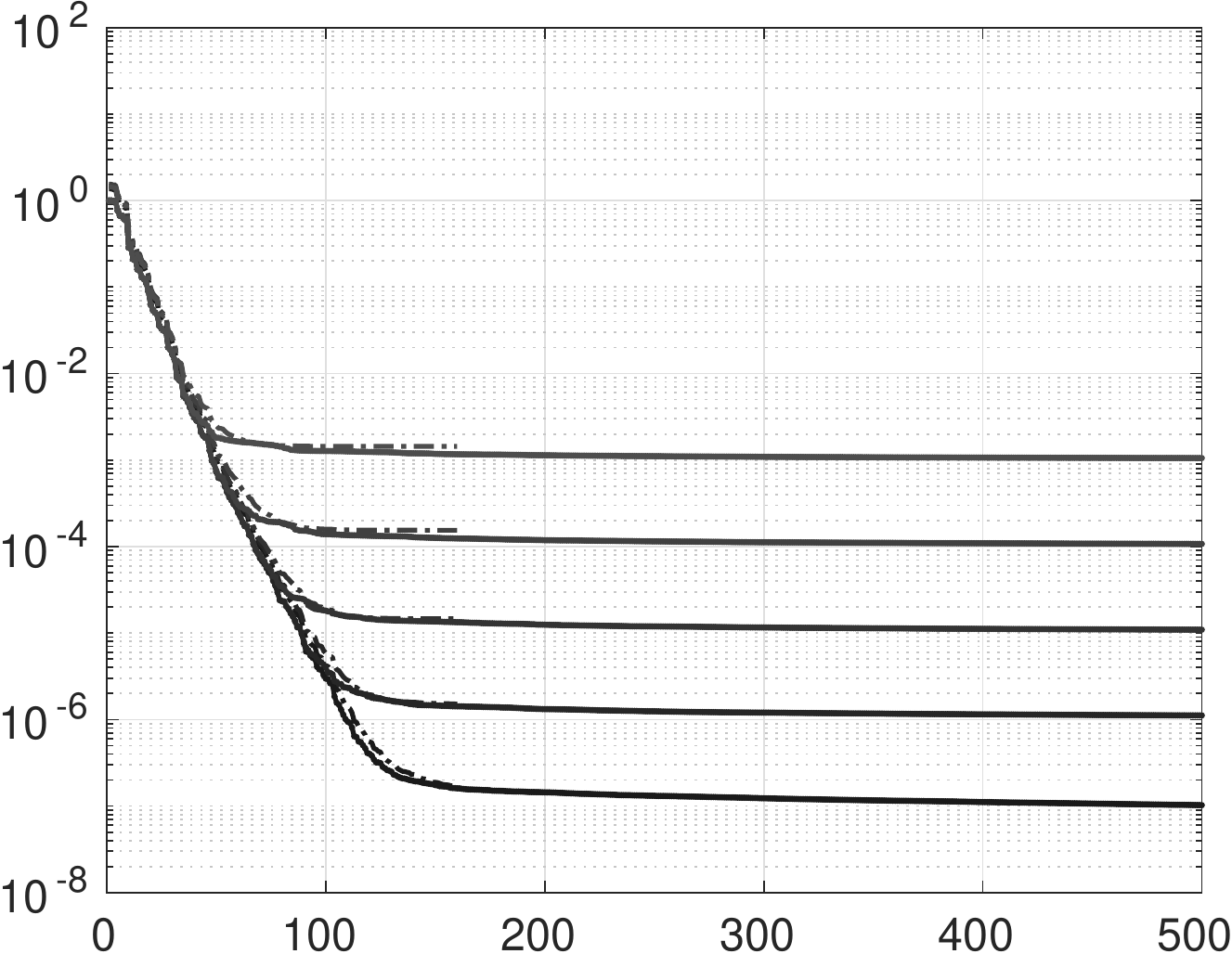}
\end{tabular}
\caption{Decay of the supremum norm of the regularized power functions $Q_n^{\lambda}$ (solid lines) obtained with the $P$-greedy algorithm for different values of the 
regularization parameters, for the Wendland (left) and Gaussian kernel (right). The dotted lines show the curves $c_{\lambda} \left(\|P_n\|_{L_{\infty}(\Omega_h)} + 
\sqrt{\lambda}\right)$, where $c_{\lambda}$ is computed as in \cref{sec:experiments}.}\label{fig:decay}
\end{figure}
Both kernels confirm the expected decay rate of \cref{th:final}, and indeed the numerically computed constant, at least for this very particular setting, seem to 
be very small.

\section*{Acknowledgements}  The authors would like to thank the German Research Foundation (DFG) for financial support within the Cluster of Excellence in
Simulation Technology (EXC 310/2) at the University of Stuttgart.

\appendix
\section{Proof of \cref{prop:DeVore_generalization}}\label{appendix}
\begin{proposition}
Let $\calh$ be a Hilbert space, $\mathcal V\subset\calh$ a subset, and 
\begin{align*}
d_n(\mathcal V, \ns) : = \inf_{\stackrel{V_n\subset\ns}{\dim(V_n) = n}} \sup_{f\in \mathcal V} 
\|f - \Pi_{V_n}(f)\| 
\end{align*}
be the Kolmogorov width of $\mathcal V$ in $\calh$. Let $\sigma_n:= \sup_{f\in \mathcal  V} \|f - \Pi_{\bar V_n}(f)\| $, where $\bar V_n$ is selected by the 
greedy algorithm of \cite{DeVore2013}. Then
\begin{enumerate}[i)]
\item\label{item:label1} $\sigma_n\leq \sqrt{2 \sigma_0} \min\limits_{1\leq m<n} d_m^{\frac{n-m}{n}}$ for all $n\in\N$.
\item\label{item:label2} If there are constants $C_0, \eta>0$ such that $d_n\leq C_0 (n^{-\alpha} + \eta)$ for all $n\in\N$, then 
$\sigma_n\leq C_1(n^{-\alpha} + \eta)\; \fa \; n\in\N$, with $C_1:=2^{1+5\alpha} C_0$.
 \item\label{item:label3} If there are constants $c_0, C_0, \eta>0$ such that $d_n\leq C_0 \left( e^{-c_0 n^{-\alpha}} + \eta\right)$ for all $n\in\N$, then 
$\sigma_n\leq C_1 \left( e^{-c_1 n^{-\alpha}} + \eta\right)\; \fa \; n\in\N$, with $C_1:=\sqrt{2 C_0 \sigma_0}$, $c_1:=2^{-1-2\alpha} c_0$.
\end{enumerate}
\end{proposition}
\begin{proof}
We use \cite[Theorem 3.2]{DeVore2013}, which states that for all $N\geq 0$, $K\geq 1$, $1\leq m <K$ it holds
\begin{equation}\label{eq:deVore}
\prod_{i=1}^K \sigma_{N+1}^2 \leq \left(\frac{K}{m}\right)^m \left(\frac{K}{K-m}\right)^{K-m} \sigma_{N+1}^{2m}\ d_m^{2 K - 2 m}. 
\end{equation}

We see the three points separately.
\begin{description}
\item[\cref{item:label1}] The proof is exactly as in \cite[Corollary 3.3]{DeVore2013}, except that $\sigma_0<1$ does not hold in general, so it is not 
simplified in the upper bound.

\item[\cref{item:label3}] Again as in \cite[Corollary 3.3]{DeVore2013}, but using \cref{item:label1}.

\item[\cref{item:label2}] This property is the only one that requires a slight modification in the estimation of the constant $C_1$, even if the other steps 
of the proof are not modified.
Since $\sigma_n$ is non increasing we have $\sigma_{2n}^{2n}\leq \prod_{j=n+1}^{2n} \sigma_j^2$, and using \cref{eq:deVore} with 
$N:=K:=n$ and $1\leq m< n$ we obtain 
\begin{align*}
\sigma_{2n}^{2n}&\leq \prod_{j=n+1}^{2n} \sigma_j^2 =\prod_{i=1}^{n} \sigma_{n +j}^2 \leq \left(\frac{n}{m}\right)^m \left(\frac{n}{n-m}\right)^{n-m} 
\sigma_{n+1}^{2m}\  d_m^{2 n - 2 m}\\
&\leq \left(\frac{n}{m}\right)^m \left(\frac{n}{n-m}\right)^{n-m} \sigma_{n}^{2m}\  d_m^{2 n - 2 m}.
\end{align*}
For $n:= 2 s$, $m :=s$ we get $\sigma_{4s}^{4s}  \leq 2^{2s} \sigma_{2s}^{2s}\  d_s^{2 s}$, i.e.,
\begin{align}\label{eq:devore_intermediate}
\sigma_{4s} & \leq \sqrt{2} \sqrt{\sigma_{2s}\  d_s}.
\end{align}
Assume it is false that $\sigma_n\leq C_1(n^{-\alpha} + \eta)$, and assume $M$ is the first index such that $\sigma_M> C_1(M^{-\alpha} + \eta)$. 

We first assume $M:=4s$, $s\geq 1$. Since the claim is true for all $n\leq M$, using \cref{eq:devore_intermediate} we obtain
\begin{align}
\sigma_{4s} & \leq \sqrt{2} \sqrt{\sigma_{2s}\  d_s} \leq \sqrt{2} \sqrt{C_1((2 s) ^{-\alpha} + \eta)\  C_0(s ^{-\alpha} + \eta)}.
\end{align}
Since we have $\sigma_M> C_1(M^{-\alpha} + \eta)$ for $M:=4s$, it follows that 
\begin{align*}
C_1((4s)^{-\alpha} + \eta) & < \sqrt{2} \sqrt{C_1((2 s) ^{-\alpha} + \eta)\  C_0(s ^{-\alpha} + 
\eta)},
\end{align*}
and dividing by $\sqrt{C_1}$ and squaring the result gives
\begin{align*}
C_1  & < \frac{2 C_0 ((2 s) ^{-\alpha} + \eta)\  (s ^{-\alpha} + \eta)}{((4s)^{-\alpha} + \eta)^2}=\frac{2 C_0 (2  ^{-\alpha} + \eta s^{\alpha})\  (1 + \eta 
s^{\alpha})}{(4^{-\alpha} + \eta s^{\alpha})^2}.
\end{align*}
Denoting as $f(s):=f_{\lambda, C_0, \alpha}(s)$ the right hand side of the last inequality, we have that  $C_1<\min_{s\geq 1} f(s)$. Since 
\begin{align*}
f'(s) & = -\frac{2^{3 \alpha +1} \left(2^{\alpha }-1\right) \alpha  C_0 \sqrt{\lambda } s^{\alpha -1} \left(2^{\alpha +1}+2^{\alpha } \left(2^{\alpha 
}+2\right) 
\sqrt{\lambda } s^{\alpha }+1\right)}{\left(4^{\alpha } \sqrt{\lambda } s^{\alpha }+1\right)^3},
\end{align*}
which is negative for $s\geq 0$ since $\alpha>0$, we can guarantee that  
\begin{align*}
C_1< f(0) = 2 C_0 2  ^{-\alpha} 4^{2\alpha} = 2^{1+3\alpha}C_0.
\end{align*}
It follows that 
\begin{align*}
C_1< 2^{1+3\alpha}C_0< 2^{1+5\alpha} C_0,
\end{align*}
which is a contradiction to our choice for $C_1$.

All the other possible cases can be covered by assuming $M:=4s + q$ with $q \in\{ 1, 2, 3\}$, $s\geq 0$. Using again \cref{eq:devore_intermediate} and the 
monotonicity of $\sigma_n$ we get
\begin{align*}
\sigma_{4s + q} \leq \sigma_{4s}\leq  \sqrt{2} \sqrt{C_1((2 s) ^{-\alpha} + \eta)\  C_0(s ^{-\alpha} + \eta)}.
\end{align*}
On the other hand, since we assumed that the bound is not valid for $n=M:=4s +q$ we have (if $s\geq 1$)
\begin{equation*}
\sigma_{4s + q} > C_1 ((4s + q)^{-\alpha} + \eta)>C_1 (2^{-\alpha}(4s)^{-\alpha} + \eta),
\end{equation*}
i.e., 
\begin{align*}
 C_1 (2^{-\alpha}(4s)^{-\alpha} + \eta) <  \sqrt{2} \sqrt{C_1((2 s) ^{-\alpha} + \eta)\  C_0(s ^{-\alpha} + \eta)},
\end{align*}
i.e.,
\begin{align*}
C_1  & < \frac{2 C_0 ((2 s) ^{-\alpha} + \eta)\  (s ^{-\alpha} + \eta)}{(2^{-\alpha}(4s)^{-\alpha} + \eta)^2}=\frac{2 C_0 (2  ^{-\alpha} + \eta s^{\alpha})\  (1 
+ \eta 
s^{\alpha})}{(2 ^{-3\alpha} + \eta s^{\alpha})^2}.
\end{align*}
In this case the derivative of the right hand side $f(s)$ is
\begin{align*}
f'(s)&=-\frac{2^{5 \alpha +1} \alpha  \left(2^{\alpha }-1\right)  \eta s^{\alpha -1} \left(2^{\alpha +1}(1  +2^{\alpha})+\eta2^{\alpha}(2^{2 
\alpha } +2  +2^{\alpha +1} ) s^{\alpha }+1\right)}{\left(8^{\alpha } \eta s^{\alpha }+1\right)^3}C_0
\end{align*}
which is again negative, so again we have that 
\begin{equation}
C_1<f(0) = \frac{2\ 2^{-\alpha } C_0}{\left(2^{-3 \alpha }\right)^2} = 2^{1+5\alpha} C_0,
\end{equation}
which is a contradiction.
\end{description}
\end{proof}

\bibliographystyle{abbrv}  
\bibliography{biblio}       

\end{document}